\documentclass[11pt, reqno]{amsart}
\usepackage{geometry}
\usepackage{amsmath, amsthm, amscd, amsfonts, amssymb, mathtools, color, hyperref}
\usepackage[dvipsnames]{xcolor}
\newtheorem{theorem}{Theorem}[section]
\newtheorem{lemma}[theorem]{Lemma}

\newtheorem{fact}[theorem]{Fact}
\theoremstyle{definition}
\newtheorem{definition}[theorem]{Definition}
\newtheorem{example}[theorem]{Example}

\theoremstyle{remark}
\newtheorem{remark}[theorem]{Remark}
\numberwithin{equation}{section}
\newcommand*{\bigchi}{\mbox{\Large$\chi$}}

\title[Matrix-valued regular functions]{Matrix-valued regular functions of a quaternionic variable - 
a decomposition and some approximation results}
\author[Dhashna and Sachindranath]{T. P. Dhashna and Sachindranath Jayaraman$^{\ast}$}
\address{School of Mathematics\\ 
Indian Institute of Science Education and Research Thiruvananthapuram\\ 
Maruthamala P.O., Vithura, Thiruvananthapuram -- 695 551, Kerala, India.}
\email{(sachindranathj,tpdhash24)@iisertvm.ac.in, sachindranathj@gmail.com}
\thanks{$^\ast$Corresponding author}


\begin{document}

\begin{abstract}
A quaternionic matrix-valued regular function is a map $F: \Omega \rightarrow M_n(\mathbb{H})$ 
whose entries are (left) regular functions of a quaternionic variable, where $\Omega$ is a domain in 
$\mathbb{H}$. Given any $n \times n$ matrix-valued regular function $F(q)$ over $\mathbb{H}$ having norm 
at most $1$, we derive a decomposition of $F(q)$ in the form $U \star \begin{bmatrix} ||F(q_0)|| & 0\\0 & G(q)\end{bmatrix} V$ for some constant unitary matrices $U, V$, some vector $q_0$ and an $(n-1) \times (n-1)$ 
matrix-valued regular function $G(q)$, where $\star$ denotes the star product over $\mathbb{H}$. We use this to 
prove a Fisher-type approximation result for matrix-valued regular functions $F(q)$ on the quaternionic unit ball 
$\mathbb{B}$: if $F(q)$ has norm at most $1$ and is continuous on the boundary of the ball, then $F(q)$ 
can be uniformly approximated by a convex combination of matrix-valued rational unitary functions on $\overline{\mathbb{B}}$.
\end{abstract}

\maketitle

\section{A brief introduction}
We work throughout over the fields $\mathbb{R}$ and $\mathbb{C}$ of real and complex 
numbers and the skew field / division ring of the set of real quaternions $\mathbb{H}$, defined by, 
$\mathbb{H}: = \{q = a_0+a_1i+a_2j+a_3k: i^2=j^2=k^2=ijk=-1, ij=-ji=k, jk=-kj=i, ki=-ik=j, 
\  a_i \in \mathbb{R}\}$. For $q = a_0+ a_1i+ a_2 j+ a_3 k \in \mathbb{H}$, the modulus and conjugate are defined by 
$|q|:=\sqrt{a_0^2 +a_1^2+a_2^2+a_3^2}$ and $\overline{q}:= a_0- a_1i -a_2 j- a_3k$, respectively. The inverse of 
a nonzero $q$ is then given by $\frac{\overline{q}}{|q|}$. Quaternions 
$p$ and $q$ are similar if $s^{-1}qs =p$ for some $s \in \mathbb{H} \setminus \{0\}$. If quaternions 
$p$ and $q$ are similar, then $|p|=|q|$. Let $\mathbb{H}^n$ denote the set of all column vectors 
with entries from the quaternions. 
For $v = \begin{bmatrix}
		v_1 \\
		\vdots \\
		v_n
	\end{bmatrix} \in \mathbb{H}^n$, $v_j \in \mathbb{H}$ and $\alpha \in \mathbb{H}$, 
let $\alpha v = 
	\begin{bmatrix}
		\alpha v_1\\
		\vdots \\
		\alpha v_n
\end{bmatrix}$. With respect to the standard addition and the scalar multiplication defined above, $\mathbb{H}^n$ is 
a left quaternion vector space over the division ring 
$\mathbb{H}$. Standard linear algebra concepts such as linear independence, 
spanning sets, basis, direct sums, and so on work in exactly the same way as for finite 
dimensional vector spaces over commutative division rings (see \cite{Rodman} for details). The subspace spanned by 
$v_1, v_2, \ldots, v_n \in \mathbb{H}^n$ is denoted by 
$\text{span}_{\mathbb{H}} \{v_1, v_2, \ldots, v_n\} = \{v_1q_1+v_2q_2+ \cdots +v_n q_n : q_i \in \mathbb{H}\}$. 
The inner product on $\mathbb{H}^n$ is defined by 
$\langle u, v \rangle = v^* u$ for $u,v \in \mathbb{H}^n$; that is, for any two vectors 
$u, v \in \mathbb{H}^n$ and $a \in \mathbb{H}$, we have 
$\langle u, v.a \rangle = \langle u, v \rangle .a$ and $\langle u.a, v \rangle 
= \bar{a}. \langle u,v \rangle$. Thus, in particular, $\mathbb{H}$ is a topological space with 
the Euclidean topology , where the norm of an element $q$ is given by $\Vert q \Vert^2 =x_0^2+x_1^2+x_2^2+x_3^2$. 

\medskip

We shall denote by $M_n(\mathbb{C})$ and $M_n(\mathbb{H})$ the rings of 
$n \times n$ matrices with complex and quaternion entries respectively. Recall that any vector 
$x = y_0 + y_1 i + y_2 j+ y_3 k \in \mathbb{H}^n$, where $y_0, y_1, y_2, \ \& \ y_3 \in \mathbb{R}^n$, 
can be uniquely expressed as $x = (y_0 + y_1 i) + (y_2 + y_3 i) j \ = \  x_1 + x_2 j$, where $x_1$ 
and $x_2 \in \mathbb{C}^n$. Similarly, any quaternion matrix $A= B_0 + B_1 i+ B_2 j+ B_3 k 
\in M_n(\mathbb{H})$, where $B_0, B_1, B_2, \ \& \ B_3 \in M_n(\mathbb{R})$,
can be uniquely expressed as $A = (B_0 + B_1 i) + (B_2 + B_3 i) j =  A_1 + A_2 j$, where $A_1$ and 
$A_2  \in M_n(\mathbb{C})$. This suggests that one can associate a complex block matrix to a given
quaternion matrix as given below.

\begin{definition}\label{complex adjoint matrix}
For $A \in M_n(\mathbb{H})$ that is expressed as $A = A_1 + A_2 j$ for some 
$A_1, A_2  \in M_n(\mathbb{C})$, the $2n \times 2n$ block matrix 
$\begin{bmatrix*}
A_1 & A_2\\
- \overline{A_2} & \overline{A_1}
\end{bmatrix*}$ is called the complex adjoint
matrix of $A$ and is denoted by $\bigchi_{A}$.
\end{definition}

It can be easily verified that $A \in M_n(\mathbb{H})$ is invertible if and only if its complex 
adjoint matrix $\bigchi_{A}$ is invertible. The Frobenius norm and the operator norm of a 
quaternion matrix are defined by 
\begin{align*}
||A||_F = \text{trace}(A^*A)^{1/2} \ \text{and} \ 
||A|| = \displaystyle \sup_{x \neq 0} \Bigl\{\dfrac{||Ax||_2}{||x||_2}\Bigr\}
\end{align*} respectively.

We end this section by observing that for $A \in M_n(\mathbb{H})$, we also have the following: 
$(1) \sqrt{2} ||A||_F = ||\bigchi_{A}||_F$ and $(2) ||A|| = ||\bigchi_{A}||$. The paper by Zhang \cite{Zhang} 
gives an excellent account of the theory of matrices over quaternions. 

We wish to highlight that the following notations will be used throughout the paper.
\begin{itemize}
\item $\ast$ will denote the conjugate transpose operation.
\item $\star$ will denote the regular product of two slice regular functions (see Definition $\ref{regular product}$).
\end{itemize}

\section{Definitions and Preliminaries}

Necessary definitions, facts, terminology and theorems used further in the paper appear 
in this section. We begin by stating that throughout the paper, a domain in $\mathbb{H}$ is an open 
connected set with respect to Euclidean topology. The unit ball in $\mathbb{H}$, centered at the origin 
will be denoted by $\mathbb{B}$.

\begin{definition}\label{2-sphere}
The $2$-sphere of quaternionic imaginary units, denoted by, $\mathbb{S}$, is defined as 
$$\mathbb{S}=\{q \in \mathbb{H} : q^2=-1\}=\{x_1 i+x_2 j +x_3 k:x_1^2+x_2^2+x_3^2=1\}.$$
\end{definition}

We shall use the above notation $\mathbb{S}$ throughout this paper to denote only the $2$-sphere. Notice that 
$\mathbb{S}$ belongs to the hyperplane $\{x_0 = 0\}$ and is therefore a $2$-dimensional sphere in $\mathbb{H}$. 
If $I \in \mathbb{S}$, then it can be easily verified that $I^2 = -1$. 
For $I \in \mathbb{S}$, we denote by $L_I$ (or $C_I$) the set $\mathbb{R} + \mathbb{R}I$, which we may 
identify with $\mathbb{C}$. We then define $\Omega_I$ to be the intersection of any domain $\Omega$ 
with $L_I$.

\begin{definition}\label{Slice-Regularity}
Let $U$ be an open subset of $\mathbb{H}$ and $f: U \rightarrow \mathbb{H}$ be a real differentiable function. $f$ is 
said to be left slice regular or left slice hyperholomorphic or simply left regular, if for every $I \in \mathbb{S}$, its restriction 
$f_I$ to $C_I$ satisfies 
$$\overline{\partial}_I f(x + Iy) = \frac{1}{2} \Bigg(\frac{\partial}{\partial x} + I \frac{\partial}{\partial y}\Bigg) 
f_I(x+Iy)I  = 0.$$
\end{definition}

An analagous definition of right regular functions also exist. \textbf{However, we work only with left regular functions 
in this paper}. It can be verified that the collection of all left regular functions 
is a right linear space over $\mathbb{H}$. It turns out that the condition of slice regularity given in Definition \ref{Slice-Regularity} holds if for a fixed $I \in \mathbb{S}$ if and only if for every $J \perp I, \ J \in \mathbb{S}$, 
there exist complex-valued holomorphic functions $F, G: U \cap C_I \rightarrow C_I$ such that $f_I(z) = F(z) + G(z)J$ 
for all $z = x+Iy \in U \cap C_I$. This equivalence comes through the {\it Splitting Lemma}, stated below.

\begin{lemma}[Splitting lemma]\label{spl}
Let $f$ be a regular function defined on a domain $\Omega$. Then for any $I \in \mathbb{S}$ and 
any $J \in \mathbb{S}$ with $I \perp J$, there exists two holomorphic functions 
$F, G:\Omega_I \rightarrow L_I$ such that for every $z = x + yI$, 
$$f_I(z) = f|_{\Omega_I}(z) = F(z) + G(z)J.$$
\end{lemma}

With the above definitions, any left polynomial over quaternions is an example of a regular function. Uniformly convergent 
power series also satisfy the conditions of Definition \ref{Slice-Regularity}. That the converse is true over balls centered 
at the origin is given in the following theorem.

\begin{theorem}\label{Slice-Regularity over Balls}
Let $B(0,R)$ be the open ball centered at $0$ of radius $R$ in $\mathbb{H}$. A function $f: B(0,R) \rightarrow \mathbb{H}$ 
is left slice regular if and only if $f$ admits a power series representation $\displaystyle \sum_{n=0}^{\infty} q^n a_n$ that 
converges uniformly on $B(0,R)$.
\end{theorem}

We thus have the following equivalent definition of left regular functions over balls centered at the origin.

\begin{definition}\label{Regular function}
A function $f: \mathbb{B} \rightarrow \mathbb{H}$ is said to be a left regular function if it admits 
a power series representation as $f(q) = \displaystyle \sum_{n=0}^{\infty} q^n f_n$, where 
$f_n \in\mathbb{H}$ and $\overline{\displaystyle{\lim_{n \to \infty}}}\sqrt[n]{|f_n|}\leq 1$.	
\end{definition}

A simple example of a non-polynomial regular function over $\mathbb{H}$ is $q^{-1}$, where 
$0 \neq q \in \mathbb{H}$. There are functions that are not regular too. We shall exhibit one in Example \ref{example-2}.

\medskip

{\bf Remark:} From now on, we use the phrase {\it regular function} only to mean that it is a {\it left regular} function, to 
avoid confusion, although some of the definitions and facts presented below also hold for right regular functions. All of the definitions, facts and theorems stated below can be found in \cite{Gal-Sabadini} or \cite{Gentili-Stoppato-Struppa}.

We now proceed to the definition of a slice domain in $\mathbb{H}$. 

\begin{definition}\label{slice domain}
Let $\mathbb{S}$ be the $2$-sphere from Definition \ref{2-sphere} and let $\Omega$ be a domain in $\mathbb{H}$ 
that intersects the real axis. $\Omega$ is called a slice domain if for all $I \in \mathbb{S}$, 
the intersection $\Omega_I = \Omega \cap L_I$ is a domain in $L_I$.
\end{definition}

Slice domains play an important role in proving results such as the identity principle as well as in 
ensuring continuity of  regular functions in a domain. 

\begin{definition}\label{axially symmetric sets}
A subset $T$ of $\mathbb{H}$ is called axially symmetric if for all points $x + yI \in T$ with 
$x,  y \in \mathbb{R}$ and $I \in \mathbb{S}$ (the $2$-sphere), the set $T$  contains the whole 
sphere $x + y\mathbb{S}$.
\end{definition}

The unit open ball $\mathbb{B}$ is an example of a symmetric slice domain in $\mathbb{H}$. 
Therefore, some of the results we prove for a symmetric slice domain will also be true when the domain 
is the unit ball $\mathbb{B}$. We now state below the identity principle for regular functions on a 
slice domain $\Omega$ (see Theorem $1.12$ of \cite{Gentili-Stoppato-Struppa}).

\begin{theorem}[Identity principle]\label{idp}
Let $f$ and $g$ be regular functions on a slice domain $\Omega$. If $f$ coincides with $g$ on a subset 
of the intersection $\Omega_I = L_I \cap \Omega$ having an accumulation point in $\Omega_I$ 
for some $I$ in the $2$-sphere $\mathbb{S}$, then $f \equiv g$ in $\Omega$.	
\end{theorem}

The following definition is that of a regular matrix-valued function.

\begin{definition}\label{Regular matrix-valued fn}
Let $F: \Omega \rightarrow M_n(\mathbb{H})$ be a quaternionic matrix-valued function defined on a 
domain $\Omega$ in $\mathbb{H}$. $F$ is said to be regular if each $ij$-entry of $F$ is a regular 
function on $\Omega$.
\end{definition}

For example, a left matrix polynomial, which is defined on the whole of $\mathbb{H}$, is a simple 
example of a regular matrix-valued function. In a recent work, Condori \cite{Condori} obtained interesting 
results concerning extensions of the maximum modulus theorem for matrix-valued analytic functions 
on a domain in $\mathbb{C}$. It turns out that the maximum modulus principle holds for regular 
functions of a quaternion variable over a slice domain (see for instance Chapter $7$ of 
\cite{Gentili-Stoppato-Struppa}). The purpose of this paper is to tackle two problems from quaternionic 
analysis: the first one is to bring out a decomposition result, very similar to the one that appears in \cite{Condori}. 
The second one concerns approximating a matrix-valued slice regular function of norm at most one on the quaternionic 
unit ball by convex combination of matrix-valued rational unitary functions.

\medskip

We now proceed to give the definitions of a few more standard notions that will be used in the later 
sections. The first of these is the definition of regular extension of $f_I$.
 
\begin{definition}\label{regular extension of f_I}
Let $\Omega$ be a symmetric slice domain and let $I \in \mathbb{S}$. If 
$f_I : \Omega_I \rightarrow \mathbb{H}$ is holomorphic, then there exists a 
unique regular function $g:\Omega \rightarrow \mathbb{H}$ such that $g_I = f_I$ in
$\Omega_I$. Such a function $g$ will be denoted by $ext(f_I)$ and is called the regular 
extension of $f_I$ and given by 
$$g(q) = ext(f_I(x + yJ)) = \frac{1}{2}\Big[f(x + yJ) + f(x - yJ) + IJ [f(x - yJ) - f(x + yJ)]\Big].$$
\end{definition}

The following three definitions (in that order) are those of the regular product (also called the star product) 
of two regular functions, regular conjugate and symmetrization of a function and the regular reciprocal of a regular function, respectively.

\begin{definition}\label{regular product}
Let $f,g$ be regular functions on a symmetric slice domain $\Omega$. 
For some $I, J \in \mathbb{S}$ with $I \perp J$, let $F, G, H, K$ be holomorphic functions from 
$\Omega_I$ to $L_I$ such that $f_I = F + GJ, \ \& \ g_I = H + KJ$. 
 The regular product of $f$ and $g$, denoted by 
$f{\star}g$, is defined to be $f{\star}g:= ext(f_I{\star}g_I)$, where 
$$f_I \star g_I(z)=[F(z)H(z)-G(z)\overline{K(\bar{z})}]+[F(z)K(z)+G(z)\overline{H(\bar{z})}]J.$$
\end{definition}

For instance, if $f(q) = \displaystyle \sum_{n=0}^{\infty} q^n a_n$ and 
$g(q) = \displaystyle \sum_{n=0}^{\infty} q^n b_n$ are two regular functions, then 
$(f \star g)(q) = \displaystyle \sum_{n=0}^{\infty} q^n \sum_{k=0}^{n} a_k b_{n-k}$.
The following fact is worth pointing out.

\begin{fact}
The collection of slice regular functions defined on a symmetric slice domain $\Omega$ forms a non-commutative ring 
with respect to $+$ and $\star$.
\end{fact}

\begin{definition}\label{regular conjugate and symmetrization}
Let $f: \mathbb{B}(0,R) \rightarrow \mathbb{H}$ be a regular function defined on a ball of radius $R$ 
and let $f(q) = \displaystyle \sum_{n \in \mathbb{N}}q^n a_n$ be its power series expansion.
The regular conjugate of $f$ is the regular function defined by 
$$f^c(q):= \displaystyle \sum_{n \in \mathbb{N}} q^n \overline{a_n}$$ on the same ball 
$\mathbb{B}(0,R)$. The symmetrization of $f$ is the function $$f^s = f \star f^c = f^c \star f.$$
\end{definition}

\begin{remark}
For symmetric slice domains, the above definition can be suitably modified, where the conjugate is defined 
as in Definition $1.3.4$ of \cite{Gentili-Stoppato-Struppa}.
\end{remark}

\begin{definition}\label{regular reciprocal}
Let $f:\Omega \rightarrow \mathbb{H}$ be a regular function. Then the regular reciprocal is the 
function $f^{- \star}:= \frac{1}{f^s}f^c$.
\end{definition}

The final definition in this section is that of the $\sigma$-ball in $\mathbb{H}$.

\begin{definition}\label{sigma ball}
For all $p,q \in \mathbb{H}$, let 
\begin{equation*}
	\sigma(p,q)=\begin{cases}
		|q-p|, \ \text{if p,q lie on the same complex line $L_I$}\\
		\sqrt{(Re(q)-Re(p))^2 + (|Im(q)|+|Im(p)|)^2}, \ \text{otherwise}.
	\end{cases}
\end{equation*}

The $\Sigma$-ball of radius $R > 0$ is the set 
$$\Sigma(p,R)=\{q \in \mathbb{H}: \sigma(q,p)<R\}.$$
\end{definition}

We state the following fact before describing the results obtained in this paper.

\begin{fact}\label{fact-1}
As pointed out in Theorem $2.12$ of \cite{Gentili-Stoppato-Struppa}, if $\Omega$ is an arbitrary domain in 
$\mathbb{H}$, then any regular function will have a power series expansion in each $\Sigma-$ball in $\Omega$. 
\end{fact}

A brief description of the results obtained are as below. After a brief introduction, the necessary 
definitions and preliminaries are brought out. This is followed by the section on main results, which is subdivided into 
two subsections for ease of reading. The first of these concerns deriving a decomposition of a matrix-valued regular 
function of norm at most $1$ (Theorem \ref{decomposition-thm}), invoving the star product over $\mathbb{H}$. 
The second part of the paper, which is quite elaborate, concerns approximating a matrix-valued regular function having 
norm at most $1$ by convex combination of rational unitary functions. The important results are the following: after introducing quaternionic Blaschke products, we point out a simple proof of Fisher's theorem over quaternions (Theorem 
\ref{Fisher-type Theorem}), which comes out as a consequence of a recent result due to Bolotnikov (Theorem \ref{Bolotnokov-Schur}). We then define matrix-valued rational unitary functions over $\mathbb{H}$ and prove a 
Fisher-type approximation theorem for matrix-valued regular functions (Theorem \ref{Fishers-Matrix}), after proving 
the necessary theorems for the same. A few remarks about singular values of matrix-valued regular functions as well as 
and some remarks / examples are made regarding the decomposition theorem.
 
\section{Main Results}

The main results are presented in this section. This section is subdivided into subsections, for 
ease of reading. Our first aim is to derive a decomposition theorem for matrix-valued slice regular functions. 
Although some of the proofs are very similar to the one presented in \cite{Condori}, we go through two of those proofs in 
detail to point out the difference between the complex and the quaternionic case. We also make a few remarks about singular 
values of matrix-valued regular functions, along with an example. This is then followed by the approximation problem 
mentioned in the introduction.

\subsection{A decomposition theorem}

This is the first section. There are three subsections, which are self-explanatory.

\subsubsection{The Frobenius and Operator norm case} \hspace*{0.5cm}

The following theorem is easily proved.

\begin{theorem}\label{Max - Frobenius norm case}
Let $\Omega$ be a domain in $\mathbb{H}$ and $F: \Omega \rightarrow M_n(\mathbb{H})$ be a quaternionic 
matrix-valued regular function. Suppose there exists a point $q_0 \in \Omega_I 
(=\Omega \cap L_I)$ for some $I \in \mathbb{S}$ such that $\Vert F(q)\Vert_F \leq F(q_0)$  
for every $q \in \Omega$. Then $F(q)$ is a constant.
\end{theorem}

\begin{proof}
The proof follows from the complex adjoint matrix representation $\widetilde{F}(q)$ and the fact that 
$\sqrt{2} ||\widetilde{F}(q)||_F = ||F(q)||_F$.
\end{proof}

The above result need not be true in general for any matrix norm. The example given in \cite{Condori} 
works in the quaternion case too. Before proceeding further we state the Hahn-Banach theorem and one of its 
consequences proved by Suhomlinov for the quaternion case (see \cite{Suhomlinov} or Theorem $4.10.1$ \ \& 
Corollary $4.10.2$ of \cite{Colombo-Sabadini-Struppa-2} for details). 

\begin{theorem}\label{Suhomlinov-thm} 
Let $L \subset \mathbb{H}$ be a left linear space over $\mathbb{H}$ and $M$ a subspace of $L$.	
Suppose	$f: M \rightarrow \mathbb{H}$ is a left linear functional whose domain is $M$ and has the 
property $|f(q)|\leq p(q)$ for all $q \in M$, where $p: L \rightarrow \mathbb{R}$ is a real valued 
functional defined on all of $L$ with the properties:
\begin{itemize}
\item[(1)] $p(q)\geq 0$,
\item[(2)] $p(q+r)\leq p(q) + p(r)$, and
\item[(3)] $p(\alpha q) = p(q)|\alpha|$.
\end{itemize} 
Then there exists an extension $F$ of $f$ to all of $L$ with the property that $|F(q)| \leq  p(q)$. 
Consequently, if $L$ is a normed linear space, $q_0$ a non-zero element of $L$, then there exists a 
continuous linear functional $T$ on $L$ such that $T(q_0 ) = ||q_0||$ and $||T||=1$.
\end{theorem}

We now use the above theorem to prove our next result, which is again along the same lines as 
that of Theorem $2$ from \cite{Condori}. 

 \begin{theorem}\label{thm2}
Let $\mathbb{B}$ be the unit ball centered at the origin in $\mathbb{H}$ and 
$F:\mathbb{B} \rightarrow M_n(\mathbb{H})$ be a quaternionic matrix-valued regular function. If $\Vert F(q) \Vert$ 
attains its maximum in $\mathbb{B}$, then $\Vert F(q) \Vert$ is a constant on $\mathbb{B}$.
\end{theorem}

\begin{proof}
Let $q_0 \in \mathbb{B}$ be such that $\Vert F(q) \Vert \leq \Vert F(q_0) \Vert$ for all $q \in \mathbb{B}$. 
Let $\Vert F(q_0)\Vert \neq 0$. From Theorem \ref{Suhomlinov-thm}, it follows that there is a bounded 
linear functional $T: M_n(\mathbb{H}) \rightarrow \mathbb{H}$ such that 
$$T(F(q_0))=\Vert F(q_0) \Vert \text{ and } \Vert T \Vert =1.$$ 
Since $F$ is left regular on $\mathbb{B}$, it admits a power series expansion as 
$$F(q)=\sum_{n=0}^{\infty} q^n F_n.$$ $T$ being a bounded (left) linear transformation, we have 
$$T(F(q)) = T\big(\sum_{n=0}^{\infty} q^n F_n\big) = \sum_{n=0}^{\infty} q^n T(F_n).$$ 
Thus, $T(F):\mathbb{B} \rightarrow \mathbb{H}$ is regular on $\mathbb{B}$, and
\begin{align*}
		|T(F(q))|& \leq \Vert F(q) \Vert\\
		& \leq \Vert F(q_0)\Vert\\
		& = T(F(q_0)).
\end{align*} 
It now follows from the maximum principle for regular functions of a quaternion variable 
(see Theorem $7.1$ of  \cite{Gentili-Stoppato-Struppa}) that $T(F(q))$ is constant on $\mathbb{B}$. 
Then, 
$$\Vert F(q) \Vert \geq |T(F(q))|=|T(F(q_0))|=\Vert F(q_0) \Vert \geq \Vert F(q) \Vert \ 
\text{for all} \ \ q \in \mathbb{B}.$$ 
Hence $\Vert F(q) \Vert$ is a constant on $\mathbb{B}$.
\end{proof}

\begin{remark}\label{rem-1}
As pointed out in Theorem $2.12$ of \cite{Gentili-Stoppato-Struppa}, if we work with an arbitrary domain 
$\Omega$ in $\mathbb{H}$, then any regular function will have a power series expansion in each $\Sigma-$ball in 
$\Omega$ and in such a case, the proof of Theorem \ref{thm2} can be modified appropriately.
\end{remark}

\subsubsection{Maximizing vectors} \hspace*{0.5cm}

We now define the notion of a maximizing vector over quaternions and prove a maximum operator norm principle for a quaternionic matrix-valued function. Throughout this section, the notation 
$\Vert \cdot \Vert$ will denote the operator norm of the matrix-valued function.

\begin{definition}\label{Maximizing Vector}
A unit vector $q_0 \in \mathbb{H}^n$ is called a maximizing vector for $A \in M_n(\mathbb{H})$ if 
$$\Vert Aq_0 \Vert_2 = \Vert A \Vert.$$
\end{definition}

$\mathbb{H}$ being a Euclidean space, the unit ball in $\mathbb{H}$ is a compact set. This assures that 
the maximizing vector always exists. We now prove that if $F(q)$ is a non constant matrix valued 
regular function whose operator norm attains its maximum at some point in $\Omega_I$, then there 
exists a vector $x_0$ such that $F(q)x_0$ is a constant. We wish to emphasize that the proof follows 
along the same line as given in \cite{Condori} for the complex case; nevertheless, it is worth going through the proof to 
stress the quaternionic arguments.

\begin{theorem}\label{thm3}
Let $\Omega$ be a slice domain in $\mathbb{H}$ and $F:\Omega \rightarrow M_n(\mathbb{H})$ 
a regular matrix-valued function. If there exists a $q_0$ in $\Omega_I$ such that 
$\Vert F(q)\Vert \leq \Vert F(q_0)\Vert$ for all $q \in \Omega$ and $x_0$ is the maximizing vector for $F(q_0)$, 
then $F^{(k)}(q_0)x_ 0 = 0$ for all $k\geq 1$.
\end{theorem}

\begin{proof}
Let the assumptions stated above hold. Choose $R > 0$ such that $\Sigma(q_0,R)$ contained in 
$\Omega$ with 
$$F(q) = \displaystyle \sum_{n \in \mathbb{N}}(q - q_0)^{\star n}F_n,$$ where $F_n \in M_n(\mathbb{H})$. 
Since $\mathbb{H}$ is a left Hilbert space, we have for any $x \in \mathbb{H}$, 
\begin{align*}
||F(q)x||_2^{2} & = ||\displaystyle \sum_{n \in \mathbb{N}}(q-q_0)^{\star n}F_n x||_2^{2}\\
& = \left \langle \displaystyle \sum_{n \in \mathbb{N}}(q - q_0)^{\star n}F_n x, 
\displaystyle \sum_{k \in \mathbb{N}}(q - q_0)^{\star k} F_k x \right\rangle\\
& = \displaystyle \sum_{n \in \mathbb{N}}\displaystyle \sum_{k \in \mathbb{N}}(q - q_0)^{\star n} 
\langle F_nx,F_kx \rangle \overline{(q-q_0)^{\star k}}.
\end{align*}

Let $\delta_I(q_0,R) = \Sigma(q_0,R) \cap L_I \subseteq \Sigma(q_0,R)$ be the disk in $L_I$ of radius $R$ and 
centered at $q_0$. Then for any $q \in \delta_I(q_0,R)$, choose $r < R$ such that $q = q_0 + re^{it}$.
We then have, 
\begin{equation*}
\displaystyle \sum_{k=0}^\infty ||F_k x||_2^{2} r^{2k}=
\frac{1}{2\pi}\int_{\partial \delta_I(q_0,R)} ||F(q_0+re^{it})x||_2^{2} dt.
\end{equation*}
The argument given in Theorem $3$ of  \cite{Condori} then yields that $F(q)x_0 = F_0x_0 = F(q_0)x_0$ for all 
$q \in \delta_I(q_0,R)$. From the identity principle (Theorem \ref{idp}), we conclude that $F(q)x_0$ is a constant 
on $\Omega$, as required.
\end{proof}

\medskip

We are now in a position to derive a decomposition theorem for matrix-valued regular functions that is very similar 
to the singular value decomposition for complex matrices. We restrict our domain to be a symmetric slice 
domain in $\mathbb{H}$. We begin with a lemma, whose proof is routine.

\begin{lemma}\label{operator norm - result}
Let $X$ be an $n \times n$ unitary matrix in $\mathbb{H}$. Then for any matrix 
$A \in M_n(\mathbb{H}), \Vert A X \Vert=\Vert A \Vert$. 
\end{lemma}

Before proceeding with deriving the decomposition result, let us point out the following. If $F(q)$ is a matrix-valued regular 
function and $U$ is a constant unitary matrix, the star product $U \star F(q)$ is again regular and $(U \star F(q))^* = 
F(q)^* \star U^*$. We now prove the aforesaid  decomposition theorem.

\begin{theorem}\label{decomposition-thm}
Let $\Omega$ be a symmetric slice domain in $\mathbb{H}$ and let 
$F:\Omega \rightarrow M_n(\mathbb{H})$ be regular. If $q_0 \in \Omega_I$ is such that 
$\Vert F(q) \Vert \leq \Vert F(q_0) \Vert$ for all $q \in \Omega$, then there are $n \times n$ 
constant unitary matrices $U$ and $V$ and a regular matrix valued function 
$G:\Omega \rightarrow M_{n-1}(\mathbb{H})$ such that 
\begin{equation*}
F(q)=U \star
\begin{bmatrix}
\Vert F(q_0) \Vert & 0\\
		0 & G(q)
\end{bmatrix} V.
\end{equation*}
\end{theorem}

\begin{proof}
Assume that $\Vert F(q_0)\Vert =1$ and let $x_0$ be the maximizing vector of $F(q_0)$. 
Since $\Vert F(q) \Vert \leq \Vert F(q_0) \Vert$, by Theorem \ref{thm3} $F(q_0)x_0$ is a constant 
on $\Omega$ and $F(q)x_0 = F(q_0)x_0$ for all $q \in \Omega$. Let us define a vector 
$y_0:= F(q_0)x_0$ such that 
$$\Vert y_0 \Vert^2=y_0^* y_0=x_0^*F(q_0)^*F(q_0)x_0=1.$$ 
Observe that $F(q)x_0 = F(q_0)x_0$ implies that $y_0^* \star F(q)x_0 = y_0^*F(q_0)x_0$ for all 
$q \in \Omega$.  Let $X_0$ and $Y_0$ be two unitary matrices whose first columns are $x_0$ and 
$y_0$ respectively. ($\mathbb{H}$ being a left Hilbert space, it is always possible to construct a 
unitary matrix from a norm one vector). Therefore we get,
\begin{equation*}
Y_0^* \star F(q) X_0=
\begin{bmatrix}
f_{11}(q) & f_{12}(q) \\
f_{21}(q) & f_{22}(q) 
\end{bmatrix}
\end{equation*} such that $f_{11}(q)=y_0^*F(q_0)x_0=1$. We then have 
$$\Vert Y_0^* \star F(q)X_0 \Vert = \Vert Y_0^* \star F(q) \Vert,$$ 
as $X_0$ is a unitary matrix. Let $H(q):= Y_0^* \star F(q)$. Since $F(q)x_0$ is a constant, we see that 
$H(q)x_0 = Y_0^* \star F(q)x_0$, which has norm one, as $Y_0$ is a unitary matrix. We thus conclude that 
$\Vert H(q) \Vert \geq 1$. However, $\Vert H(q) \Vert \leq 1$ as $\Vert F(q) \Vert \leq 1$. We thus conclude 
that $\Vert Y_0^* \star F(q) X_0 \Vert = 1$. The operator norm of an $n \times n$ matrix being an upper bound 
on the euclidean norm of its columns and rows, we have $ f_{12}(q)=f_{21}(q)=0$. Therefore, 
\begin{equation*}
		F(q)=Y_0 \star
		\begin{bmatrix}
			1 & 0 \\
			0 & f_{22}(q) 
		\end{bmatrix}X_0^*,
\end{equation*} 
where $f_{22}(q)$ is a regular $(n-1) \times (n-1)$ matrix-valued function. We set 
$U=Y_0, V=X_0^*$ and $G(q):= f_{22}(q)$ to get the desired conclusion.
\end{proof}

\medskip
\noindent
A few remarks are in order.

\begin{remark}\label{some remarks}
\
\begin{itemize}
\item Note that for $n=1$, the function $F:\Omega \rightarrow M_1(\mathbb{H}) \cong \mathbb{H}$ is a 
regular function. Suppose there exists $q_0 \in \Omega_I$ such that $\Vert F(q)\Vert \leq \Vert F(q_0) \Vert$, 
for all $q \in \Omega$, then by maximum modulus principle (Theorem $7.1$ of \cite{Gentili-Stoppato-Struppa}), 
it follows that $F$ is a constant on $\Omega$. By Theorem \ref{decomposition-thm} we can write $F(q)$ as $$F(q) = u \Vert F(q_0)\Vert v$$ where 
$||u||=||v||=1$. Let $x_0$ be the maximizing vector of $F(q_0)$ so that 
$$\Vert F(q_0)\Vert = \sup_{\Vert x \Vert_2 =1}\Vert F(q_0)x\Vert_2 = \Vert F(q_0)x_0\Vert_2.$$ 
Then $u = x_0, v = F(q_0)x_0$, so that $F(q) = x_0 \Vert F(q_0)\Vert F(q_0)x_0$.

\item Let $n = 2$ and $F: \Omega \rightarrow M_2(\mathbb{H})$ a regular function. Suppose there 
exists $q_0 \in \Omega_I$ such that $\Vert F(q)\Vert \leq \Vert F(q_0) \Vert$ so that a decomposition 
as stated in Theorem \ref{decomposition-thm} holds. If there exists a sequence of non-vanishing regular functions 
$g_n: \Omega \to \mathbb{H}$ which converges uniformly to $g: \Omega \to \mathbb{H}$, 
then by Hurwitz theorem for regular functions, $g$ is either identically zero or non-vanishing (see Theorem $4.5$ 
of \cite{Mongodi}). In the former case, the function $F$ becomes a constant. Notice that there is a sequence 
of polynomials on $\Omega$ that converges uniformly to a given regular function on $\Omega$ 
(see for instance Chapter $3$ of \cite{Gal-Sabadini}).
\end{itemize}
\end{remark}

The following example shows that above decomposition holds only if such a vector $q_0$ exists in the 
domain $\Omega$.

\begin{example}\label{example-1}
Let $r$ and $R$ be positive reals numbers with $r < R$ and let $\Omega$ be the annular region 
defined by $\Omega:= \{q\in \mathbb{H} :r < |q| <R\}$. Consider the quaternionic matrix-valued 
function $F$ defined on $\Omega$ by $F(q) = \begin{bmatrix} 
																						 q-1 & 2\\
																						 0 & \dfrac{1}{q}
																						\end{bmatrix}$. Since $q \neq 0$, $F$ is regular. 
By calculating the norm of $F(q)$, we get
$$\Vert F(q) \Vert = \frac{\sqrt{|q-1|^2+4+\frac{1}{|q|^2}+\sqrt{(|q-1|^2+4+
\frac{1}{|q|^2})^2-4|q-1|^2}}}{2}.$$ 
It is then clear that $\Vert F(q)\Vert$ is non-constant, whose maximum is obtained at a point $q_0$ 
such that $|q_0| = r$. Suppose $F(q)$ is decomposed as stated in the Theorem \ref{decomposition-thm}.  
Then, there exists constant unitary matrices $U \ \& \ V$ and a regular function 
$g: \Omega \rightarrow \mathbb{H}$ such that 
$$F(q) = U \star \begin{bmatrix}
		             \Vert F(q_0) \Vert & 0\\
		             0 & g(q)
	                \end{bmatrix}V^* = U \star A(q) V^*.$$ 
As in the proof of Theorem \ref{decomposition-thm}, $\Vert F(q) \Vert = \Vert A(q) \Vert 
= \max \{\Vert F(q_0) \Vert, \Vert g(q)\Vert \}$.
If $\Vert F(q_0) \Vert  > \Vert g(q)\Vert$, then $\Vert F(q) \Vert = \Vert A(q) \Vert
= \Vert F(q_0) \Vert$. This is not possible as $\Vert F(q) \Vert$ is non-constant. On the otherhand, 
if $\Vert g(q)\Vert  > \Vert F(q_0) \Vert $, then again we have 
$\Vert F(q) \Vert \leq \Vert F(q_0) \Vert$ for all $q \in \Omega$, which in turn implies that 
$\Vert F(q) \Vert < \Vert g(q)\Vert = \Vert F(q)\Vert $ which is 
again a contradiction. Thus, a decomposition as stated in Theorem \ref{decomposition-thm} does not 
hold for $F(q)$, thereby elucidating that the point $q_0$ must be inside the domain $\Omega$ in 
Theorem \ref{decomposition-thm}.
\end{example}

\medskip

\begin{remark}\label{important rem about q_0}
It should be pointed out that we have assumed that the point $q_0 \in \Omega_I$ (the point where 
the maximum is attained for $F(q)$) for some $I \in \mathbb{S}$ in the theorems proved in 
this section and the previous one. However, this is merely a technical issue and the proofs will go through 
verbatim, with appropriate modifications, if $q_0 \in \Omega$. This is because for any 
$q \in \mathbb{H}$, we can choose $I, J \in \mathbb{S}$ such that $q \in L_I$ by taking 
$v = x_1i + x_2j + x_3k$ (the imaginary part of $q$) and by letting $I = \dfrac{v}{|v|}$ so that 
$q = x_0 + |v|I \in L_I$.
\end{remark}

\subsubsection{Some remarks about singular values of matrix-valued regular functions} \hspace*{0.5cm} 

It is a known fact that singular values and the singular value decomposition of complex matrices have numerous 
applications. Singular values of a quaternion matrix $A$ are defined to be the positive square 
roots of the eigenvalues of $A^*A$, where $A^*$ denotes the conjugate transpose of $A$. It is also not 
hard to verify that the singular values of the complex adjoint matrix $\widetilde{A}$ are repeated twice, 
and hence the singular values of $A$ are taken to be the same as it appears for the complex adjoint matrix, 
without repetition. The SVD of quaternion matrices also finds applications in areas such as image processing 
(see for instance \cite{S. Pei-Chang-Ding}; also refer to \cite{Rodman} for details on SVD of quaternion matrices).
This helps us to make a few remarks about a maximum principle for a matrix-valued regular function 
$F$, when the singular values of $F(q)$ attains its maximum in $\Omega_I$. In what follows, we 
shall denote the singular values of $F(q)$ by $s_k(F(q))$. It should be emphasized that these are defined merely as a 
real-valued function of a quaternionic variable. With this in mind, we state below a maximum principle for singular values 
of $F(q)$.

\begin{theorem}\label{maximum princ - svd}
Let $\Omega$ be a slice domain in $\mathbb{H}$ and $F:\Omega \rightarrow M_n(\mathbb{H})$ a 
regular matrix-valued function. If for each $k=1, 2, \cdots, n$, the function $q \rightarrow s_k(F(q))$
attains its maximum value on $\Omega_I$, then $F(q)$ is a constant on $\Omega$.
\end{theorem}

The proof follows from the splitting lemma and then from the complex adjoint representation of $F(q)$. We therefore skip the proof and proceed to make a few remarks.

\medskip

Regularity of the entries of the matrix-valued function cannot be dispensed with in the above theorem, as the following 
example illustrates. 

\begin{example}\label{example-2}
Let $F: \mathbb{B} \rightarrow M_n(\mathbb{H})$ be defined by 
$F(q) = \begin{bmatrix}
		      x_0 + x_2j & 0\\
		     0 & 1
	       \end{bmatrix}$. Notice that $F$ is not a regular function as the entry $x_0 + x_2j$ is not regular. 
The singular values of $F$ can be easily evaluated to be $\sqrt{x_0^2 + x_2^2}$ and $1$ respectively. 
In this case, the function $s_k(q): \mathbb{B} \rightarrow \mathbb{H}$ for $k \in \{1,2\}$ (the singular 
value of $F(q)$) has maximum value as $1$ for any $q \in \mathbb{H}$. However, the function $F(q)$ 
is a non-constant function.
\end{example}

\medskip

We end by pointing out that the singular values of a matrix-valued function cannot attain its minimum 
value at a point in the domain unless the matrix-valued function is not invertible at that point. We again skip the proof, 
as it is along similar lines to that of Theorem $9$ of \cite{Condori}.

\begin{theorem}\label{svd-min-inv}
Let $F:\Omega \rightarrow M_n(\mathbb{H})$ be a non-constant regular function on a symmetric slice domain 
$\Omega$ in $\mathbb{H}$. Suppose there exist $q_0 \in \Omega_I$ such that for all 
$k=1, \cdots, n$, the singular values $s_k(F(q))$ achieves its minimum at $q_0$. Then $F(q_0)$ is not invertible.
\end{theorem}

\subsection{Approximation of a quaternionic matrix-valued function} \hspace*{0.5cm}

This is the second section of the paper. We prove as an application of Theorem \ref{decomposition-thm}, 
a Fisher type approximation theorem for matrix-valued regular functions. This section is again subdivided into further subsections.

\subsubsection{Fisher's theorem over quaternions - norm convergence}

We begin this section by defining quaternionic 
Blaschke factors (see \cite{Bolotnikov-1} for details). Other definitions and terminologies will be introduced 
as we proceed further. Recall that the quaternionic Hardy space $H^2$ of the unit ball 
$\mathbb{B}$ is defined as the space of square summable regular power series:
$$H^2:= \displaystyle \left\{f(q) =  \sum_{n=0}^{\infty} q^n f_n: \Vert f \Vert^2_{H^2}:= 
\sum_{n=0}^{\infty} |f_n|^2<\infty \right\}.$$

For $\alpha \in \mathbb{B}$, the power series $k_\alpha(q) = \displaystyle \sum_{n=0}^{\infty} q^n \alpha^n$ 
is in $H^2$ and $\Vert k_\alpha \Vert^2_{H^2}=\displaystyle \frac{1}{1-|\alpha|^2}$. 

\begin{definition}\label{Finite-Blaschke}
The power series $b_\alpha(q):= k_{\bar{\alpha}}(q) \star (q-\alpha)$ is defined as the quaternionic analog of a 
Blaschke factor. The power series $$B(q):=b_{\alpha_1}(q) \star b_{\alpha_2}(q) \star \cdots \star b_{\alpha_n}(q) \phi, 
\text{ where } |\phi|=1, \alpha_i \in \mathbb{B} \text{ for all } 1\leq i \leq n$$ is called a finite Blaschke product.
\end{definition} 

Note that this notion coincides with the usual definition of a Blaschke product over the field of complex numbers. 
Let us recall the following complex approximation theorem attributed to Fisher (see Theorem $4.2.2$, \cite{Garcia-Mashreghi-Ross} as well as the first theorem in \cite{Fisher}).

\begin{theorem}\label{Fisher-Complex}
Each $f$ in the unit ball of the disc algebra $\mathcal{A}(\mathbb{D})$ can be uniformly approximated on $\overline{\mathbb{D}}$ by a convex combination of finite Blaschke products.
\end{theorem}

As pointed out in the introduction, our objective is to prove a matrix-valued version of the above theorem. 
Let us also point out the following theorem (Theorem $4.10$) from Bolotnikov \cite{Bolotnikov-2}.

\begin{theorem}\label{Bolotnokov-Schur}
Let $f \in \mathcal{S}_{\mathbb{H}}$. For any $\rho < 1$ and $\epsilon > 0$, there exists a finite Blaschke
product $B$, such that 
$|f^{e_\ell}(\alpha) - B^{e_\ell}(\alpha)| < \epsilon$ and $|f^{e_r}(\alpha) - B^{e_r}(\alpha)| < \epsilon$.
\end{theorem}

Therefore, if $f$ is a regular function on the quaternionic unit ball $\mathbb{B}$ taking values in the closure of 
$\mathbb{B}$ and continuous on the boundary of the ball, then $f$ can be approximated by a sequence of 
finite Blaschke products uniformly on every compact subset of the ball. We use the above theorem to give a 
simple proof of norm convergence of finite Blaschke products for such functions on the quaternionic unit ball.

\begin{theorem}\label{Fisher-type Theorem}
Let $f: \mathbb{B} \rightarrow \overline{\mathbb{B}}$ be a regular function that is also continuous on the boundary 
of the ball. Then $f$ can be approximated in norm by a sequence of convex combinations of finite Blaschke products 
over $\mathbb{H}$.
\end{theorem}

\begin{proof}
Note that the radial function $f_t(q):= f(tq), t \in [0,1)$ is regular and continuous on the boundary of the unit ball of 
$\mathbb{H}$. The conclusion follows from Theorem \ref{Bolotnokov-Schur} and the proof of Fisher's theorem in 
the complex case.
\end{proof}

\subsubsection{Matrix-valued rational unitary functions}

The following definition and theorems are from \cite{Alpay-Colombo-Sabadini-2} and \cite{Alpay-Colombo-Sabadini-4}. We begin with the following definition.

\begin{definition}
A matrix-valued regular function $F(q) $ defined on an axially symmetric open set $\Omega$ is called a 
rational function if the restriction of $F$ to the real line is a rational function.
\end{definition}

 Let us point out that the terminology \textit{hyperholomorphic} in the following statements simply means that the function is regular.

\begin{definition}\label{J-Unitary}
Let $J$ be signature matrix ($J=J^*=J^{-1}$). An matrix-valued function slice hyperholomorphic function $\Theta$ 
is said to be $J-$unitary if $\Theta(x) J \Theta(-x)^* = J$, at all points where it is defined.
\end{definition}

Note that in the case of unit ball, the above condition is equivalent to $\Theta(x)J\Theta(1/x)^* = J$.

\begin{theorem}
Let $U$ be an axially symmetric open set with $0 \in U$. The function $\Theta$ is a slice hyperholomorphic rational 
function on $U$, with realization $\Theta(q) = A + qB \star (I - qD)^{- \star} C$ if and only if $A$ is $J$-unitary, and 
there exists a uniquely determined Hermitian matrix $H$ such that 
\begin{align*}
	B & = AJC^*H\\
	HD + D^*H & = B^*JB.
\end{align*}
\end{theorem}

When $J = I$, rational matrix-valued unitary functions which are also slice hyperholomorphic, are precisely finite 
Blaschke products. In what follows, we work only with rational unitary functions and approximate a regular matrix-valued 
function of norm at most $1$ on the quaternionic unit ball by a convex combination of rational matrix-valued unitary 
functions. We wish to point out that in a recent paper, rational unitary functions over scaled quaternions were studied 
in \cite{Alpay-Cho-Vajiac-2}.

The following definition is that of a Blaschke-Potapov factor over quaternions, as given in \cite{Alpay-Colombo-Sabadini-3}. 

\begin{definition}\label{Blaschke-Potapov}
An $n \times n$ matrix-valued Blaschke-Potapov factor of fist kind is defined as $B_a(q,P):= I_n + (b_a(q) - I_n)P$ 
where $|a|<1, \ b_a(q) = (1-q\bar{a})^{- \star} \star (a-q)\frac{a}{|a|}$ and $P \in M_n(\mathbb{H})$ is such 
that $P^2 = P = P^*$.
\end{definition}

We now state below three well known results (in order) in quaternionic analysis, one of which gives equivalent conditions 
for a matrix-valued function to be a rational function. Let us also mention that these equivalent statements were also 
studied in the context of scaled quaternions (see Theorem $7.2$ of \cite{Alpay-Cho-Vajiac}).

\begin{theorem}\label{State-space Thm}
Let $A, B, C, D$ be given matrices of appropriate size with quaternionic entries. Suppose $\{u_n\}_{n \in \mathbb{N}_0}$ 
is a given sequence of vectors with quaternionic entries, and of appropriate size. Consider the system
\begin{align}
x_{n+1}&=Dx_n+Cu_n, n=0,1,\cdots\\
y_n&=Bx_n+Au_n, 
\end{align}
and define its slice hyperholomorphic (regular) matrix-valued transfer function as 
$$\mathcal{H}(q):=\mathcal{Y}(q)\star (\mathcal{U}(q))^{-\star}$$ where $\mathcal{Y}(q)$ and $\mathcal{U}(q)$ 
are the slice hyperholomorphic extensions of the $Z-$transforms of $y_n$ and $u_n$, respectively. Then we have 
\begin{equation}
\mathcal{H}(q) = A + qB\star (I-qD)^{-\star}C.
\end{equation}
\end{theorem}

\begin{theorem}\label{ratioanl-equivlence}
Let $F$ be a matrix-valued function, slice hyperholomorphic in a neighbourhood of the origin. Then, the following are 
equivalent: \\
(i) There exist matrices $A, B$ and $C$, of appropriate dimensions, and such that
$$F_n = BD^{n-1}C, n=1,2, \cdots, \text{ and }F_0=A.$$
(ii) $F$ can be written as $$F(q)=A+qB\star (I-qD)^{-\star}C.$$ 
(iii) The entries of $F$ are of the form $P\star Q^{-\star}$, where $P$ and $Q$ are slice hyperholomorphic polynomials, and $Q(0)\neq 0$.
\end{theorem}

The following theorem, which appears as Theorem $8.3.1$ of \cite{Alpay-Colombo-Sabadini-2}, is written by taking 
$\mathcal{P}_1, \mathcal{P}_2$ and $\mathcal{P}$ to be $\mathbb{H}^d$ and $\mathbb{H}^n$, respectively instead of Pontryagin spaces.

\begin{theorem}\label{Unit-Rat}
Let $\mathcal{P}_1 = \mathcal{P}_2 = \mathbb{H}^d$ and $\mathcal{P} = \mathbb{H}^n$. Assume that the operator 
matrix, called colligation operator, 
$$U=\begin{bmatrix}
A & B\\C & D
\end{bmatrix}:\begin{bmatrix}
\mathcal{P}\\ \mathcal{P}_1
\end{bmatrix} \rightarrow \begin{bmatrix}
\mathcal{P}\\ \mathcal{P}_2
\end{bmatrix}$$ is a contraction. Then the function $$S(q)= A+qB \star (I_n - qD)^{-\star}C$$ 
is a slice hypermeromorphic in $\mathbb{B}$.
\end{theorem}

\subsubsection{Matrix-valued Fisher-type theorem}

We are now in a position to prove the remaining (approximation) theorems of the paper.

\begin{theorem}\label{rat-approx}
Any rational function $F:\mathbb{B} \rightarrow M_n(\mathbb{H})$ with $\Vert F(q)\Vert \leq 1$ for all 
$q \in \mathbb{B}$ can be uniformly approximated on compact subsets of $\mathbb{B}$ by matrix-valued rational 
unitary functions.
\end{theorem}

\begin{proof}
By Theorem \ref{ratioanl-equivlence} any matrix-valued rational function can be written as 
$$F(p)=A+pB \star (I-pD)^{-\star}C,$$ where $A, B, C, D $ are matrices of appropriate dimensions.
Consider the contraction $T$ defined by 
$$T = \begin{bmatrix}
A & B \\ C & D
\end{bmatrix}:\mathbb{H}^n \oplus \mathbb{H}^d \rightarrow \mathbb{H}^n \oplus \mathbb{H}^d,$$ 
where $d$ is the McMillan degree of the matrix-valued rational function $F$. Let $D_{T^*}=\begin{bmatrix}
S_1 & S_2\\S_3 & S_4 \end{bmatrix}$ and $D_T=\begin{bmatrix}
T_1 & T_2\\T_3 & T_4 \end{bmatrix}$ be the defect operators $(I-TT^*)^{1/2}$ and $(I-T^*T)^{1/2}$ respectively, from $\mathbb{H}^n \oplus \mathbb{H}^d$ into itself. Set $\mathcal{T}:=\mathbb{H}^n \oplus \mathbb{H}^d$ and 
consider the block operator matrix 
$$U_m:=\begin{bmatrix}
T & 0 & \cdots & 0 & D_{T^*}\\
D_T & 0 & \cdots & 0 & -T^*\\
0_\mathcal{T} & I_\mathcal{T} & \cdots &0 & 0_\mathcal{T}\\
\vdots & \vdots & \ddots & \vdots & \vdots  \\ 
0_\mathcal{T} & 0 &\cdots& I_\mathcal{T} & 0_\mathcal{T}
\end{bmatrix}$$ acting on $(m+1)$ copies of $\mathcal{T}$, $\mathcal{T}\oplus \cdots \oplus \mathcal{T}$ are $(m+1)$. It is easily verified that $U_m$ is a unitary matrix. Then, an argument analagous to the complex case 
(see \cite{Alpay-Thirtha-Jindal-Kumar}), the block operator 
matrix $U_m$ can be written as $\begin{bmatrix}
A & B_m\\C_m & D_m \end{bmatrix}$, where \\
$B_m=\begin{bmatrix}
B & 0 & \cdots & 0 & S_1 & S_2 \end{bmatrix}$, 
$C_m=\begin{bmatrix}
C & T_1 & T_3 & 0 & \cdots & 0 \end{bmatrix}^t$ and \\ 
$D_m=\begin{bmatrix}
D & 0 & \cdots & 0 & S_3 & S_4\\
T_2 & 0 & \cdots & 0 & -A^* & -C^*\\
T_4 & 0 & \cdots & 0 & -B^* & -D^*\\
0 & I_\mathcal{T} & \cdots & 0 & 0 & 0\\
\vdots & \vdots & \ddots & \vdots & \vdots & \vdots  \\
0 & 0 & \cdots & I_\mathcal{T} & 0 & 0
\end{bmatrix}$. \\ 
Clearly, $D_m^kC_m=\begin{bmatrix}
D^kC & * & * & * & \cdots & 0 & 0 \end{bmatrix}^t$, which in turn implies that $B_mD_m^kC_m=BD^kC$ for all 
$m \geq k+2$. By Theorem \ref{Unit-Rat}, we can associate a matrix-valued rational function 
$F_m(q) = A + qB_m \star (I-qD_m)^{-\star}C_m$. Notice that $F_m$ is a matrix-valued rational unitary function as 
$U_m$ is a unitary matrix (also see Theorem $9.3.2$ of \cite{Alpay-Colombo-Sabadini-2}). Consider $q \in \mathbb{B}$ 
such that $\vert q \vert \leq \rho < 1$ and choose $M_0$ with $\rho^{M_0}< \epsilon$, where $\epsilon > 0$ is arbitrary. 
Then, the following calculations hold: 
 
 \begin{align*}
 \Vert F(q)-F_m(q)\Vert
 &=\Vert (A+qB\star (I-qD)^{-\star}C) - (A+qB_m\star (I-qD_m)^{-\star}C_m) \Vert\\
 &=\Vert qB\star (I-qD)^{-\star}C-qB_m\star (I-qD_m)^{-\star}C_m\Vert\\
 &=\Vert qB\star \bigg(\displaystyle \sum_{k \geq 0}^{\infty }q^kD^k \bigg) C-qB_m \star 
 \bigg(\displaystyle \sum_{k \geq 0}^{\infty}q^kD_m^k \bigg)C_m\Vert\\
 &=\Vert q \displaystyle\sum_{k\geq 0}^{\infty}q^k(B_mD_m^kC-BD^kC)\Vert \\
 &\quad \text{\small (since all products are $\star$-products)}\nonumber\\
 &\leq\vert q \vert \sum_{k\geq M_0-1} q^k\Vert B_mD_m^kC_m-BD^kC\Vert 
 \ \text{\small (as $B_mD_m^kC_m=BD^kC$ for $m \geq k+2$)}\\
 &\leq\vert q \vert \sum_{k\geq M_0-1}2 q^k \ 
 \text{\small (since $\Vert B_mD_m^kC_m\Vert =\Vert BD^kC \Vert\leq 1$ )}\\
 &=\epsilon \frac{\vert q\vert}{1-\vert q \vert}
 \end{align*}
 
We thus conclude that $F$ can be approximated by a sequence of matrix-valued rational unitary functions $F_m$.
\end{proof}

The following theorem is a quaternionic version of a result proved in the complex case by Knese (see Lemma $6.2$ of  \cite{Knese}.)

\begin{theorem}\label{poly-approx}
Let $F:\mathbb{B} \rightarrow M_n(\mathbb{H})$ be regular and $\Vert F(q) \Vert <1$ for all $q \in \mathbb{B}$. 
Then for any $r \in (0,1)$ and $\epsilon > 0$, there exists a quaternionic matrix-valued regular polynomial $P$ such 
that $\Vert P\Vert_\mathbb{B} < 1$ and $\Vert F - P\Vert_{r \mathbb{B}} < \epsilon$.
\end{theorem}

\begin{proof}
Let $F_r(q):= F(rq)$ for $r \in (0,1)$. Note that $\Vert F_r\Vert<1$. Since $F_r$ is continuous on $\bar{\mathbb{B}}$ 
it is uniformly continuous on $\bar{\mathbb{B}}$. Therefore, by uniform continuity, 
$\Vert F_{r} - F_{rs} \Vert < \frac{\epsilon}{2}$. Using regularity of $F$, we can choose a regular polynomial $P$ of 
$F_s$ such that $\Vert F_s - P\Vert < \min(1-\Vert F_s\Vert, \frac{\epsilon}{2})$. By triangle inequality, we have
$$\Vert P\Vert \leq \Vert F_s-P \Vert +\Vert F_s\Vert<1-\Vert F_s\Vert + \Vert F_s\Vert.$$
This then implies that $\Vert P \Vert < 1$ and $\Vert F_r - P_r \Vert \leq \Vert F_r - F_{rs}\Vert + 
\Vert F_{rs} - P_r \Vert < \epsilon$.
\end{proof}

We now use Theorem \ref{decomposition-thm} to prove the following approximation theorem.

\begin{theorem}\label{reg-approx}
Let $F:\mathbb{B} \rightarrow M_n(\mathbb{H})$ be matrix-valued regular function such that $\Vert F(q)\Vert \leq 1$ 
for all $q \in \mathbb{B}$. Then $F$ can be approximated uniformly on compact subsets of $\mathbb{B}$ by 
matrix-valued rational unitary functions.
\end{theorem}

\begin{proof}
There are two possible cases. Either $\Vert F(q)\Vert < 1$ or $\Vert F(q)\Vert = 1$. If $\Vert F(q)\Vert < 1$, by Theorem \ref{poly-approx} we can approximate $F$ by a matrix-valued regular polynomial. Since any matrix valued regular 
polynomial is itself a rational function, it can be approximated by matrix-valued rational unitary functions uniformly 
on compact subsets of the ball.

\smallskip
\noindent
If $\Vert F(q)\Vert=1$, then using Theorem \ref{decomposition-thm}, we can get two constant unitary matrices $U, V$ 
and an  $M_{n-1}(\mathbb{H})-$valued regular function $G$ such that,
\begin{equation}
F(q) = U \star \begin{bmatrix} 1 & 0\\ 0 & G(q) \end{bmatrix} V.
\end{equation}
When $n=2$, the result follows from Theorem \ref{Bolotnokov-Schur}. Let us assume that the result is true for $n=k-1$. For $n=k$, by Theorem \ref{decomposition-thm}, we get a matrix-valued regular function $G(q)$ 
of size $(n-1) \times (n-1)$, which has norm at most $1$. If $||G(q)|| < 1$, then proceed as in the first step above. If $||G(q)|| = 1$, then by the induction hypothesis, $G(q)$ 
can be approximated uniformly on compact subsets of $\mathbb{B}$ by matrix-valued rational unitary 
functions. It then easily follows that $F(q)$ is approximated uniformly on compact subsets of $\mathbb{B}$. 
This proves the theorem.
\end{proof}

The final result of this paper is to prove a Fisher-type theorem for matrix-valued regular functions. We set out to 
do this by proving some preliminary results. Let us recall the definition of $F_r$ for $r \in (0,1)$ from 
Theorem \ref{poly-approx}. We observe that 

\begin{equation}\label{reg-dil}
(F\star G)_r(q) = (F\star G)(rq) = F(rq) \star G(rq) = F_r(q) \star G_r(q).
\end{equation}

We prove below two lemmas that are very similar to the ones stated in the complex case (see for instance \cite{Alpay-Thirtha-Jindal-Kumar}).

\begin{lemma}\label{reg-rat}
Let $F$ and $G$ be two matrix-valued rational unitary functions. Assume that for a fixed $r = [0,1]$, if $F_r, G_r$ can be 
written as convex combinations of rational unitary functions. Then $(F\star G)_r$ can also be written as a convex 
combination of rational unitary functions. 
\end{lemma}

\begin{proof}
Let $\{R_i\}_{i=1}^n$ and $\{T_j\}_{j=1}^m$ be matrix-valued rational unitary functions such that 
$F_r(q) = \displaystyle \sum_{i=1}^{n}\lambda_i R_i(q)$ and $G_r(q) = \displaystyle \sum_{j=1}^{m}\chi_j T_j(q)$, 
where $\displaystyle \sum_{i=1}^{n}\lambda_i= \sum_{j=1}^{m}\chi_j=1$. 
By Equation \ref{reg-dil} we have 
\begin{align*}
(F \star G)_r(q)
&=F_r(q)\star G_r(q)\\
&=\sum_{i=1}^{n}\lambda_i R_i(q) \star \sum_{j=1}^{m}\chi_j T_j(q)\\
&=\sum_{i=1}^{n}  \lambda_i \sum_{j=1}^{m}\chi_j (R_i\star T_j)(q)
\end{align*}
where $\sum_{j=1}^{m}\chi_j\sum_{i=1}^{n}\lambda_i=1$. The last equality in the above calculation follows from 
Fact \ref{fact-1}. It is left to show that $R_i\star T_j$ is a rational unitary function for each $1\leq i \leq n$ and 
$1\leq j \leq m$. Clearly it is rational as each of the entries of $R_i\star T_j$ are rational, being a sum of 
a regular product of rational functions. Moreover, if $R$ and $T$ are matrix-valued rational unitary funtions, then
$$(R(x)\star T(x)) (R(-x)\star T(-x))^* = (R(x)\star (T(x))\star (T(-x)^*) \star R(-x)^*) = I.$$  
This is true for all real $x$ (wherever it is defined) and hence $(F\star G)_r(q)$ is a convex combination of rational unitary functions.
\end{proof}

\begin{lemma}\label{unit-rat}
Suppose a matrix-valued rational unitary function $F$ is such that for a fixed $r \in [0,1]$, $F_r$ can be written as a convex combination of rational unitary functions. Then, for any constant unitary matrix $U$ over $\mathbb{H}$, $U \star F_r$ can also be written as a convex combination of rational unitary functions.
\end{lemma}

\begin{proof}
Let $\{R_i\}_{i=1}^n$ be a matrix-valued rational unitary functions such that 
$F_r(q) = \displaystyle \sum_{i=0}^{n}\lambda_i R_i(q)$ where $\sum_{i=0}^{n}\lambda_i =1$. Consider
$U \star F_r(q) = \displaystyle \sum_{i=0}^{n} \lambda_i U \star R_i(q)$. Since $R_i$ is a matrix-valued rational 
function, so is $U \star R_i$. It is also a unitary function as 
$$((U \star R_i)(x)) \star ((U \star R_i)(-x))^*= 
U \star (R_i(x) R_i(-x)^*)U^* = I,$$ for all real $x$. This yields the desired conclusion. Note that the second equality follows from the fact that $(U \star R_i)(x)^* = (R_i(x)^* U^*)$ for each $x$.
\end{proof}

\begin{lemma}\label{rat-unitary-approx}
Let $F$ be any matrix-valued rational unitary function and $U$ be a constant unitary matrix over $\mathbb{H}$ of 
appropriate size. Then, for any $r\in (0,1)$, $F_r$ can be written as a convex combination of matrix-valued rational 
unitary functions. 
\end{lemma}

\begin{proof}
If $F$ is a scalar-valued rational unitary function and $P$ is a projection onto some subspace, then the matrix-valued 
function $FP + (I - P)$ is also a rational unitary function. Then, by Theorem \ref{Fisher-type Theorem} (see also Theorem 
$4.10$ of \cite{Bolotnikov-2}), any $(b_a)_r(q)$ can be written as a convex combination of scalar-valued 
rational unitary 
functions, where $(b_a)_r(q) = b_a(rq)$. Then, $(b_a)_r P - (I - P)$, which equals $(B_a)_r(q,P)$, can be written as a 
convex combination of matrix-valued rational unitary functions. However, any matrix-valued rational unitary function 
can be written as $U \star B_{a_{1}}(q,P_1) \star \dots \star B_{a_{k}}(q,P_k)$, for some 
unitary matrix $U$, $|a_i| < 1$ and $P_i^2 = P_i$. The desired conclusion then follows from Lemmas \ref{reg-rat} and \ref{unit-rat}.
\end{proof}

Before proving the final result of this paper, let us record the following fact, whose proof follows very much similar to that 
of the complex case (see Lemma $3.4$ of \cite{Alpay-Thirtha-Jindal-Kumar}).

\begin{lemma}\label{dil-reg-approx}
Let $F$ be a matrix-valued regular function on the quaternionic unit ball $\mathbb{B}$ and continuous on 
$\bar{\mathbb{B}}$. Then $F_r$ converges uniformly to $F$ on $\bar{\mathbb{B}}$ as $r \to 1$.
\end{lemma}

We are now in a position to prove a Fisher-type approximation theorem for matrix-valued regular functions $F(q)$ having 
norm at most one.

\begin{theorem}\label{Fishers-Matrix}
Let $F$ be a matrix-valued regular function that is also continuous on $\bar{\mathbb{B}}$. Assume further that 
$\Vert F(q)\Vert \leq 1$ for all $q \in \mathbb{B}$. Then $F$ can be uniformly approximated on $\bar{\mathbb{B}}$ by 
a convex combination of matrix-valued rational unitary functions.
\end{theorem}

\begin{proof}
For $\epsilon > 0$ and $r \in (0,1)$, we have from Lemma \ref{dil-reg-approx} 
$$\Vert F - F_r \Vert_{\bar{\mathbb{B}}} < \frac{\epsilon}{2}.$$ Let $\mathbb{B}_r$ be a closed ball of radius $r$. 
By Theorem \ref{reg-approx}, it is possible to approximate $F$ by a matrix-valued rational unitary function on 
$\mathbb{B}_r$, say $\Theta$, so that  
$$\Vert F-\Theta\Vert_{\mathbb{B}_r}<\frac{\epsilon}{2}.$$ 
We also have $\Vert F_r-\Theta_r\Vert_{\bar{\mathbb{B}}} < \frac{\epsilon}{2}$. Then, triangle inequality gives 
\begin{align*}
\Vert F-\Theta_r\Vert_{\bar{\mathbb{B}}}\leq \Vert F-F_r\Vert_{\bar{\mathbb{B}}}+ \Vert F_r-\Theta_r\Vert_{\bar{\mathbb{B}}}
&\leq \epsilon.
\end{align*}
To finish the proof, we resort to Lemma \ref{rat-unitary-approx}, which ensures that $\Theta_r$ can be approximated by 
a convex combination of matrix-valued rational unitary functions.
\end{proof}

\section{Concluding Remarks}
We derive a decomposition theorem for matrix-valued regular functions $F(q)$ of a quaternionic variable that are 
uniformly bounded by one, in the form 
$U \star \begin{bmatrix} 1 & 0\\0 & G(q)\end{bmatrix} V$ for some constant unitary matrices $U, V$, and an 
$(n-1) \times (n-1)$ matrix-valued regular function $G(q)$, where $\star$ denotes the star product over $\mathbb{H}$. 
We use this decomposition to prove that any matrix-valued regular function $F(q)$ 
having norm at most one and continuous on the boundary of the unit ball $\mathbb{B}$, can be approximated uniformly by convex combinations of matrix-valued rational unitary functions on $\overline{\mathbb{B}}$. 
We also make some remarks on maximum principles for singular values of matrix-valued regular functions. Examples illustrating the validity of the assumption in the decomposition theorem as well for singular values are included.

\medskip
\noindent

{\bf Declarations:} 
\begin{itemize}
\item Both T. P. Dhashna and Sachindranath Jayaraman have equal contribution in this work and declare that there is no conflict of interest.
\item No funding was received for conducting this study.
\item The authors also declare that there is no data used in this research.
\end{itemize}

\bibliographystyle{amsplain}

\end{document}